\documentclass[twoside,reqno,11pt]{amsart}

\usepackage{latexsym}

\newcommand{\qdn}{\hspace*{-1.5mm}}
\newcommand{\qqdn}{\hspace*{-2.5mm}}
\newcommand{\xqdn}{\hspace*{-5.0mm}}
\newcommand{\xxqdn}{\hspace*{-10mm}}

\newcommand{\fns}{\footnotesize}

\newcommand{\fnk}[3]{\left[\qdn\ba{#1}#2\\[0.8mm]#3\ea\qdn\right]}
\newcommand{\ffnk}[4]{\left[\qdn\ba{#1}#3\\[0.8mm]#4\ea{\!\Big|\:#2}\right]}

\newcommand{\binm}{\binom}

\newcommand{\nnm}{\nonumber}
\newcommand{\be}{\begin{equation}}
\newcommand{\ee}{\end{equation}}
\newcommand{\ba}{\begin{array}}
\newcommand{\ea}{\end{array}}
\newcommand{\bmn}{\begin{eqnarray}}
\newcommand{\emn}{\end{eqnarray}}
\newcommand{\bnm}{\begin{eqnarray*}}
\newcommand{\enm}{\end{eqnarray*}}
\newcommand{\bln}{\begin{subequations}}
\newcommand{\eln}{\end{subequations}}

\newtheorem{thm}{Theorem}

\newtheorem{corl}[thm]{Corollary}
\newtheorem{prop}[thm]{Proposition}

\newtheorem{entry}{Entry}

\newcommand{\bbtm}[4]{\bibitem{kn:#1}{#2,}~{#3,}~{#4.}}
\newcommand{\cito}[1]{\cite{kn:#1}}
\newcommand{\citu}[2]{\cite[#2]{kn:#1}}

\begin{document} 
{\fns
\title{Extensions of two $q$-Gosper identities \\with an extra integer parameter}

\author{Chuanan Wei$^1$, Qinglun Yan$^2$}
\dedicatory{
$^1$Department of Information Technology\\
  Hainan Medical College, Haikou 571199, China\\
  $^2$College of Mathematics and Physics\\
   Nanjing University of Posts and Telecommunications,
    Nanjing 210046, China}
\thanks{\emph{Email addresses}:
      weichuanan78@163.com (C. Wei), yanqinglun@126.com (Q. Yan)}

\address{ }


\begin{abstract}
According to the method of series rearrangement, we establish the
extensions of two $q$-Gosper identities with an extra integer
parameter. The limiting cases of them produce the generalizations of
Gosper's two $_3F_2(\frac{3}{4})$-series identities with an
additional integer parameter. Meanwhile, several related results are
also given.
\end{abstract}

\maketitle\thispagestyle{empty}
\markboth{Chuanan Wei, Qinglun Yan}
         {Extensions of two $q$-Gosper identities with an extra integer parameter}

\section{Introduction}
For a complex variable $x$, define the shifted-factorial by
 \bnm
 (x)_n=
\begin{cases}
\prod_{k=0}^{n-1}(x+k),&n>0;\\
1,&n=0;\\
 \frac{(-1)^n}{\prod_{k=1}^{-n}(k-x)},&n<0.
\end{cases}
 \enm
 For simplifying the expressions, we shall use the abbreviated symbol:
\[\qqdn\qdn\fnk{ccccc}{a,&b,&\cdots,&c}{\alpha,&\beta,&\cdots,&\gamma}_n
=\frac{(a)_n(b)_n\cdots(c)_n}{(\alpha)_n(\beta)_n\cdots(\gamma)_n}.\]
 Following Bailey~\cito{bailey}, define the hypergeometric series by
\[_{1+r}F_s\ffnk{cccc}{z}{a_0,&a_1,&\cdots,&a_r}{&b_1,&\cdots,&b_s}
 \:=\:\sum_{k=0}^\infty\fnk{ccccc}{a_0,&a_1,&\cdots,&a_r}{1,&b_1,&\cdots,&b_s}_nz^k,\]
where $\{a_{i}\}_{i\geq0}$ and $\{b_{j}\}_{j\geq1}$ are complex
parameters such that no zero factors appear in the denominators of
the summand on the right hand side. Then two
$_3F_2(\frac{3}{4})$-series identities due to Gosper (cf.
\citu{gessel}{Equations (1.4) and (1.5)}) read, respectively, as
 \bmn
&&\qqdn_3F_2\ffnk{cccc}{\frac{3}{4}}{3x,&1-3x,&-n}{&\frac{1}{2},&-3n}
=\fnk{ccccc}{\frac{1}{3}+x,&\frac{2}{3}-x}{\frac{1}{3},&\frac{2}{3}}_{n},
 \label{gosper-a}\\\label{gosper-b}
&&\qqdn_3F_2\ffnk{cccc}{\frac{3}{4}}{3x,&2-3x,&-n}{&\frac{3}{2},&-1-3n}
=\fnk{ccccc}{\frac{2}{3}+x,&\frac{4}{3}-x}{\frac{2}{3},&\frac{4}{3}}_{n}.
 \emn

 For two complex numbers $x$ and
$q$, define the $q$-shifted factorial by
  \bnm
 (x;q)_n=
\begin{cases}
\prod_{i=0}^{n-1}(1-xq^i),&n>0;\\
\quad1,&n=0;\\
\frac{1}{\prod_{j=n}^{-1}(1-xq^j)},&n<0.
\end{cases}
\enm
 The fraction form of it reads as
\[\qqdn\qdn\ffnk{ccccc}{q}{a,&b,&\cdots,&c}{\alpha,&\beta,&\cdots,&\gamma}_n
=\frac{(a;q)_n(b;q)_n\cdots(c;q)_n}{(\alpha;q)_n(\beta;q)_n\cdots(\gamma;q)_n}.\]
 Following Gasper and Rahman \cito{gasper}, define the $q$-series by
\[_{1+r}\phi_s\ffnk{cccc}{q;z}{a_0,&a_1,&\cdots,&a_r}{&b_1,&\cdots,&b_s}
 \!=\!\sum_{k=0}^\infty\ffnk{ccccc}{q}{a_0,&a_1,&\cdots,&a_r}{q,&b_1,&\cdots,&b_s}_k
 \Big\{(-1)^kq^{\binm{k}{2}}\Big\}^{s-r}z^k,\]
where $\{a_i\}_{i\geq0}$ and $\{b_j\}_{j\geq1}$ are complex
parameters such that no zero factors appear in the denominators of
the summand on the right hand side. Then the q-analogues of
\eqref{gosper-a} and \eqref{gosper-b} due to Chu
\citu{chu-a}{Equations (3.9a) and (3.9b)} can respectively be stated
as
  \bmn
&&\qqdn\sum_{k=0}^n(q^{-3n};q^3)_k\ffnk{ccccc}{q}{x,q/x}
{q,-q,q^{1/2},-q^{1/2},q^{-3n}}_kq^k=\ffnk{ccccc}{q^3}{qx,q^2/x}{q,q^2}_n,
 \label{q-gosper-a}\\\label{q-gosper-b}
&&\qqdn\sum_{k=0}^n(q^{-3n};q^3)_k\ffnk{ccccc}{q}{x,q^2/x}
{q,-q,q^{3/2},-q^{3/2},q^{-1-3n}}_kq^k=\ffnk{ccccc}{q^3}{q^2x,q^4/x}{q^2,q^4}_n.
 \emn

 Inspired by the work of \cito{chu-b}-\cito{lavoie-b}, we shall
 establish, in terms of the method of series rearrangement,
 the extensions of \eqref{q-gosper-a} and \eqref{q-gosper-b} with an extra integer
parameter which involve the generalizations of \eqref{gosper-a} and
\eqref{gosper-b} with an additional integer parameter in Section 2.
Meanwhile, several related results are also offered in Section 3.

\section{Extensions of two $q$-Gosper identities}

\begin{thm}\label{thm-a}
For a nonnegative integer $\ell$ and a complex number $x$, there
holds
 \bnm
&&\xqdn\sum_{k=0}^n(q^{-3n};q^3)_k\ffnk{ccccc}{q}{x,q^{1-\ell}/x}
{q,-q^{1-\ell},q^{1/2},-q^{1/2},q^{-3n}}_kq^k=\ffnk{ccccc}{q}{x,-x}{x^2,-1}_{\ell}
\\&&\xqdn\:\:\times\:\sum_{i=0}^{\ell}(-1)^iq^{\ell i}\frac{1-x^2q^{2i-1}}{1-x^2q^{-1}}
\ffnk{ccccc}{q}{q^{-\ell},x^2q^{-1}}{q,x^2q^{\ell}}_i\ffnk{ccccc}{q^3}{q^{1+i}x,q^{2-i}/x}{q,q^2}_{n}.
 \enm
\end{thm}

\begin{proof}
Setting $a=x^2q^{-1}$, $b=xq^{k}$ and $c=-x$ in
 the terminating $_6\phi_5$-series identity (cf. \citu{gasper}{p.
 42}):
 \bmn\label{terminating-65}
\:\:\qquad {_6\phi_5}\ffnk{cccccccc}{q;\frac{q^{1+\ell}a}{bc}}
{a,\:q\sqrt{a},\:-q\sqrt{a},\:b,\:c,\:q^{-\ell}}
 {\sqrt{a},\:-\sqrt{a},\:qa/b,\:qa/c,\:q^{1+\ell}a}
 =\ffnk{ccccc}{q}{qa,qa/bc}{qa/b,qa/c}_\ell,
 \emn
 we get the equation:
 \bmn
&&\ffnk{ccccc}{q}{q^{1-\ell}/x,-q^{1-\ell}/x}{q^{1-\ell}/x^2,-q^{1-\ell+k}}_{\ell}
 \sum_{i=0}^{\ell}(-1)^iq^{\ell i}\frac{1-x^2q^{2i-1}}{1-x^2q^{-1}}
\ffnk{ccccc}{q}{q^{-\ell},xq^k,x^2q^{-1}}{q,x,x^2q^{\ell}}_i
 \nnm\\\label{series}&&\times\:
\ffnk{ccccc}{q}{q^{1-\ell+k}/x}{q^{1-\ell}/x}_{\ell-i}=1.
 \emn
 Then there is the following relation:
  \bnm
&&\xqdn\sum_{k=0}^n(q^{-3n};q^3)_k\ffnk{ccccc}{q}{x,q^{1-\ell}/x}
{q,-q^{1-\ell},q^{1/2},-q^{1/2},q^{-3n}}_kq^k
\\&&\xqdn\:\:=\:\sum_{k=0}^n(q^{-3n};q^3)_k\ffnk{ccccc}{q}{x,q^{1-\ell}/x}
{q,-q^{1-\ell},q^{1/2},-q^{1/2},q^{-3n}}_kq^k
\ffnk{ccccc}{q}{q^{1-\ell}/x,-q^{1-\ell}/x}{q^{1-\ell}/x^2,-q^{1-\ell+k}}_{\ell}
\\&&\xqdn\:\:\times\:
\sum_{i=0}^{\ell}(-1)^iq^{\ell i}\frac{1-x^2q^{2i-1}}{1-x^2q^{-1}}
\ffnk{ccccc}{q}{q^{-\ell},xq^k,x^2q^{-1}}{q,x,x^2q^{\ell}}_i
\ffnk{ccccc}{q}{q^{1-\ell+k}/x}{q^{1-\ell}/x}_{\ell-i}.
 \enm
Interchange the summation order for the last double sum to achieve
  \bnm
&&\xxqdn\sum_{k=0}^n(q^{-3n};q^3)_k\ffnk{ccccc}{q}{x,q^{1-\ell}/x}
{q,-q^{1-\ell},q^{1/2},-q^{1/2},q^{-3n}}_kq^k
\\&&\xxqdn\:\:=\ffnk{ccccc}{q}{q^{1-\ell}/x,-q^{1-\ell}/x}{q^{1-\ell}/x^2,-q^{1-\ell}}_{\ell}
\sum_{i=0}^{\ell}(-1)^iq^{\ell i}\frac{1-x^2q^{2i-1}}{1-x^2q^{-1}}
\ffnk{ccccc}{q}{q^{-\ell},x^2q^{-1}}{q,x^2q^{\ell}}_i
\\&&\xxqdn\:\:\times\:
\sum_{k=0}^n(q^{-3n};q^3)_k\ffnk{ccccc}{q}{xq^i,q^{1-i}/x}
{q,-q,q^{1/2},-q^{1/2},q^{-3n}}_kq^k.
 \enm
 Calculating the series on the last line by
\eqref{q-gosper-a}, we attain Theorem \ref{thm-a} to complete the
proof.
\end{proof}

When $\ell=0$, Theorem \ref{thm-a} reduces to \eqref{q-gosper-a}.
Another concrete formula is displayed as follows.

\begin{corl}[$\ell=1$ in Theorem \ref{thm-a}]\label{corl-a}
 \bnm
&&\xqdn\sum_{k=0}^n(q^{-3n};q^3)_k\ffnk{ccccc}{q}{x,1/x}
{q,-1,q^{1/2},-q^{1/2},q^{-3n}}_kq^k
\\&&\xqdn\:\:=\;
\frac{1}{2}\ffnk{ccccc}{q^3}{qx,q^{2}/x}{q,q^2}_{n}+\frac{1}{2}\ffnk{ccccc}{q^3}{q^2x,q/x}{q,q^2}_{n}.
 \enm
\end{corl}

Performing the replacement $a\to q^{3x}$ in Theorem \ref{thm-a} and
then letting $q\to1$, we obtain the following equation.

\begin{prop}\label{prop-a}
For a nonnegative integer $\ell$ and a complex number $x$, there
holds
 \bnm
&&\:\xxqdn_3F_2\ffnk{cccc}{\frac{3}{4}}{3x,&1-\ell-3x,&-n}{&\frac{1}{2},&-3n}
=\fnk{ccccc}{3x}{6x}_{\ell}
\\&&\:\xxqdn\:\:=\:\:\sum_{i=0}^{\ell}(-1)^i\frac{6x+2i-1}{6x-1}\fnk{ccccc}{-\ell,6x-1}{1,6x+\ell}_i
\fnk{ccccc}{\frac{1+i}{3}+x,&\frac{2-i}{3}-x}{\frac{1}{3},&\frac{2}{3}}_{n}.
 \enm
\end{prop}

When $\ell=0$, Proposition \ref{prop-a} reduces to \eqref{gosper-a}.
Another concrete formula is expressed as follows.

\begin{corl}[$\ell=1$ in Proposition \ref{prop-a}]\label{corl-b}
\bnm
 \qquad_3F_2\ffnk{cccc}{\frac{3}{4}}{3x,&-3x,&-n}{&\frac{1}{2},&-3n}
 =\frac{1}{2}\fnk{ccccc}{\frac{1}{3}+x,&\frac{2}{3}-x}{\frac{1}{3},&\frac{2}{3}}_{n}
+\frac{1}{2}\fnk{ccccc}{\frac{2}{3}+x,&\frac{1}{3}-x}{\frac{1}{3},&\frac{2}{3}}_{n}.
\enm
\end{corl}

\begin{thm}\label{thm-b}
For a nonnegative integer $\ell$ and a complex number $x$, there
holds
 \bnm
&&\xqdn\sum_{k=0}^n(q^{-3n};q^3)_k\ffnk{ccccc}{q}{x,q^{2-\ell}/x}
{q,-q,q^{3/2},-q^{3/2-\ell},q^{-1-3n}}_kq^k=\ffnk{ccccc}{q}{xq^{-1},-xq^{-1/2}}{x^2q^{-1},-q^{-1/2}}_{\ell}
\\&&\xqdn\:\:\times\:\sum_{i=0}^{\ell}(-1)^iq^{(\ell-\frac{1}{2}) i}\frac{1-x^2q^{2i-2}}{1-x^2q^{-2}}
\ffnk{ccccc}{q}{q^{-\ell},x,x^2q^{-2}}{q,xq^{-1},x^2q^{\ell-1}}_i\ffnk{ccccc}{q^3}{q^{2+i}x,q^{4-i}/x}{q^2,q^4}_{n}.
 \enm
\end{thm}

\begin{proof}
 Taking $a=x^2q^{-2}$, $b=xq^{k}$ and $c=-xq^{-1/2}$ in \eqref{terminating-65}, we gain the equation:
 \bnm
&&\xqdn\ffnk{ccccc}{q}{q^{2-\ell}/x,-q^{3/2-\ell}/x}{q^{2-\ell}/x^2,-q^{3/2-\ell+k}}_{\ell}
 \sum_{i=0}^{\ell}(-1)^iq^{(\ell-\frac{1}{2}) i}\frac{1-x^2q^{2i-2}}{1-x^2q^{-2}}
\ffnk{ccccc}{q}{q^{-\ell},xq^k,x^2q^{-2}}{q,xq^{-1},x^2q^{\ell-1}}_i
 \\&&\xqdn\times\:
\ffnk{ccccc}{q}{q^{2-\ell+k}/x}{q^{2-\ell}/x}_{\ell-i}=1.
 \enm
 Then we can proceed as follows:
  \bnm
&&\xqdn\sum_{k=0}^n(q^{-3n};q^3)_k\ffnk{ccccc}{q}{x,q^{2-\ell}/x}
{q,-q,q^{3/2},-q^{3/2-\ell},q^{-1-3n}}_kq^k
\\&&\xqdn\:\:=\:\sum_{k=0}^n(q^{-3n};q^3)_k\ffnk{ccccc}{q}{x,q^{2-\ell}/x}
{q,-q,q^{3/2},-q^{3/2-\ell},q^{-1-3n}}_kq^k
\ffnk{ccccc}{q}{q^{2-\ell}/x,-q^{3/2-\ell}/x}{q^{2-\ell}/x^2,-q^{3/2-\ell+k}}_{\ell}
\\&&\xqdn\:\:\times\:
\sum_{i=0}^{\ell}(-1)^iq^{(\ell-\frac{1}{2})
i}\frac{1-x^2q^{2i-2}}{1-x^2q^{-2}}
\ffnk{ccccc}{q}{q^{-\ell},xq^k,x^2q^{-2}}{q,xq^{-1},x^2q^{\ell-1}}_i
\ffnk{ccccc}{q}{q^{2-\ell+k}/x}{q^{2-\ell}/x}_{\ell-i}.
 \enm
Interchange the summation order for the last double sum to achieve
  \bnm
&&\xxqdn\sum_{k=0}^n(q^{-3n};q^3)_k\ffnk{ccccc}{q}{x,q^{2-\ell}/x}
{q,-q,q^{3/2},-q^{3/2-\ell},q^{-1-3n}}_kq^k
\\&&\xxqdn\:\:=\ffnk{ccccc}{q}{q^{2-\ell}/x,-q^{3/2-\ell}/x}{q^{2-\ell}/x^2,-q^{3/2-\ell}}_{\ell}
\sum_{i=0}^{\ell}(-1)^iq^{(\ell-\frac{1}{2})i}\frac{1-x^2q^{2i-2}}{1-x^2q^{-2}}
\ffnk{ccccc}{q}{q^{-\ell},x,x^2q^{-2}}{q,xq^{-1},x^2q^{\ell-1}}_i
\\&&\xxqdn\:\:\times\:
\sum_{k=0}^n(q^{-3n};q^3)_k\ffnk{ccccc}{q}{xq^i,q^{2-i}/x}
{q,-q,q^{3/2},-q^{3/2},q^{-1-3n}}_kq^k.
 \enm
 Evaluating the series on the last line by
\eqref{q-gosper-b}, we attain Theorem \ref{thm-b} to finish the
proof.
\end{proof}

When $\ell=0$, Theorem \ref{thm-b} reduces to \eqref{q-gosper-b}.
Another concrete formula is displayed as follows.

\begin{corl}[$\ell=1$ in Theorem \ref{thm-b}]\label{corl-c}
 \bnm
&&\xqdn\sum_{k=0}^n(q^{-3n};q^3)_k\ffnk{ccccc}{q}{x,q/x}
{q,-q,q^{3/2},-q^{1/2},q^{-1-3n}}_kq^k
\\&&\xqdn\:\:=
\frac{1-x}{(1+q^{1/2})(1-xq^{-1/2})}\ffnk{ccccc}{q^3}{q^3x,q^{3}/x}{q^2,q^4}_{n}
\\&&\xqdn\:\:+\:
\frac{1-xq^{-1}}{(1+q^{-1/2})(1-xq^{-1/2})}\ffnk{ccccc}{q^3}{q^2x,q^{4}/x}{q^2,q^4}_{n}.
 \enm
\end{corl}

Employing the substitution $a\to q^{3x}$ in Theorem \ref{thm-b} and
then letting $q\to1$, we obtain the following equation.

\begin{prop}\label{prop-b}
For a nonnegative integer $\ell$ and a complex number $x$, there
holds
 \bnm
&&\:\xxqdn_3F_2\ffnk{cccc}{\frac{3}{4}}{3x,&2-\ell-3x,&-n}{&\frac{3}{2},&-1-3n}
=\fnk{ccccc}{3x-1}{6x-1}_{\ell}
\\&&\:\xxqdn\:\:\times\:\:\sum_{i=0}^{\ell}(-1)^i\bigg(\frac{3x+i-1}{3x-1}\bigg)^2\fnk{ccccc}{-\ell,6x-2}{1,6x+\ell-1}_i
\fnk{ccccc}{\frac{2+i}{3}+x,&\frac{4-i}{3}-x}{\frac{2}{3},&\frac{4}{3}}_{n}.
 \enm
\end{prop}

When $\ell=0$, Proposition \ref{prop-b} reduces to \eqref{gosper-b}.
Another concrete formula is expressed as follows.

\begin{corl}[$\ell=1$ in Proposition \ref{prop-b}]\label{corl-d}
\bnm
 &&3F_2\ffnk{cccc}{\frac{3}{4}}{3x,&1-3x,&-n}{&\frac{3}{2},&-1-3n}\\
 &&\:\:=\:\frac{3x}{6x-1}\fnk{ccccc}{1+x,&1-x}{\frac{2}{3},&\frac{4}{3}}_{n}
 +\frac{3x-1}{6x-1}\fnk{ccccc}{\frac{2}{3}+x,&\frac{4}{3}-x}{\frac{2}{3},&\frac{4}{3}}_{n}.
\enm
\end{corl}

\begin{prop}\label{prop-c}
For a nonnegative integer $\ell$ and a complex number $x$, there
holds
 \bnm
&&\:\xxqdn_3F_2\ffnk{cccc}{\frac{3}{4}}{3x,&2-\ell-3x,&-n}{&\frac{3}{2}-\ell,&-1-3n}
=\fnk{ccccc}{3x-1,3x-\frac{1}{2}}{6x-1,-\frac{1}{2}}_{\ell}
\\&&\:\xxqdn\:\:\times\:\:\sum_{i=0}^{\ell}\bigg(\frac{3x+i-1}{3x-1}\bigg)^2\fnk{ccccc}{-\ell,6x-2}{1,6x+\ell-1}_i
\fnk{ccccc}{\frac{2+i}{3}+x,&\frac{4-i}{3}-x}{\frac{2}{3},&\frac{4}{3}}_{n}.
 \enm
\end{prop}

\begin{proof}
 Performing the replacement $q\to q^{2}$ in Theorem \ref{thm-b}, we
 have
 \bnm
&&\xqdn\sum_{k=0}^n(q^{-6n};q^6)_k\ffnk{ccccc}{q^2}{x,q^{4-2\ell}/x}
{q^2,-q^2,q^{3},-q^{3-2\ell},q^{-2-6n}}_kq^{2k}=\ffnk{ccccc}{q^2}{xq^{-2},-xq^{-1}}{x^2q^{-2},-q^{-1}}_{\ell}
\\&&\xqdn\:\:\times\:\sum_{i=0}^{\ell}(-1)^iq^{(2\ell-1) i}\frac{1-x^2q^{4i-4}}{1-x^2q^{-4}}
\ffnk{ccccc}{q^2}{q^{-2\ell},x,x^2q^{-4}}{q^2,xq^{-2},x^2q^{2\ell-2}}_i\ffnk{ccccc}{q^6}{q^{4+2i}x,q^{8-2i}/x}{q^4,q^8}_{n}.
 \enm
 Replace $q$ by $-q^{1/2}$ to gain
  \bnm
&&\xqdn\sum_{k=0}^n(q^{-3n};q^3)_k\ffnk{ccccc}{q}{x,q^{2-\ell}/x}
{q,-q,q^{3/2-\ell},-q^{3/2},q^{-1-3n}}_kq^k=\ffnk{ccccc}{q}{xq^{-1},xq^{-1/2}}{x^2q^{-1},q^{-1/2}}_{\ell}
\\&&\xqdn\:\:\times\:\sum_{i=0}^{\ell}q^{(\ell-\frac{1}{2}) i}\frac{1-x^2q^{2i-2}}{1-x^2q^{-2}}
\ffnk{ccccc}{q}{q^{-\ell},x,x^2q^{-2}}{q,xq^{-1},x^2q^{\ell-1}}_i\ffnk{ccccc}{q^3}{q^{2+i}x,q^{4-i}/x}{q^2,q^4}_{n}.
 \enm
Employing the substitution $a\to q^{3x}$ in the last equation and
then letting $q\to1$, we acquire Proposition \ref{prop-c}.
\end{proof}

When $\ell=0$, Proposition \ref{prop-c} also reduces to
\eqref{gosper-b}. Another concrete formula is expressed as follows.

\begin{corl}[$\ell=1$ in Proposition \ref{prop-c}]\label{corl-e}
\bnm
&&_3F_2\ffnk{cccc}{\frac{3}{4}}{3x,&1-3x,&-n}{&\frac{1}{2},&-1-3n}\\
 &&\:\:=\:3x\fnk{ccccc}{1+x,&1-x}{\frac{2}{3},&\frac{4}{3}}_{n}
+(1-3x)\fnk{ccccc}{\frac{2}{3}+x,&\frac{4}{3}-x}{\frac{2}{3},&\frac{4}{3}}_{n}.
\enm
\end{corl}

\section{Several related results}
\begin{thm}\label{thm-c}
For a nonnegative integer $\ell$ and a complex number $x$, there
holds
 \bnm
&&\xqdn\sum_{k=0}^n(q^{-3n};q^3)_k\ffnk{ccccc}{q}{x,q^{-\ell}/x}
{q,-1,q^{1/2-\ell},-q^{1/2},q^{-3n}}_kq^k
\\&&\xqdn\:\:=
 \ffnk{ccccc}{q}{qx,q^{1/2}x}{qx^2,q^{1/2}}_{\ell}\sum_{i=0}^{\ell}q^{(\ell+\frac{1}{2})i}\frac{1+xq^{i}}{2(1+x)}
\ffnk{ccccc}{q}{q^{-\ell},x^2}{q,x^2q^{\ell+1}}_i
\\&&\xqdn\:\:\times\:
\Bigg\{\ffnk{ccccc}{q^3}{q^{1+i}x,q^{2-i}/x}{q,q^2}_{n}+\ffnk{ccccc}{q^3}{q^{2+i}x,q^{1-i}/x}{q,q^2}_{n}\Bigg\}.
 \enm
\end{thm}

\begin{proof}
Setting $a=x^2$, $b=xq^{k}$ and $c=q^{1/2}x$ in
\eqref{terminating-65}, we get the equation:
 \bnm
&&\ffnk{ccccc}{q}{q^{-\ell}/x,q^{1/2-\ell}/x}{q^{-\ell}/x^2,q^{1/2-\ell+k}}_{\ell}
 \sum_{i=0}^{\ell}q^{(\ell+\frac{1}{2})i}\frac{1-x^2q^{2i}}{1-x^2}
\ffnk{ccccc}{q}{q^{-\ell},xq^k,x^2}{q,qx,x^2q^{\ell+1}}_i
 \\&&\times\:
\ffnk{ccccc}{q}{q^{k-\ell}/x}{q^{-\ell}/x}_{\ell-i}=1.
 \enm
 Then there is the following relation:
  \bnm
&&\xqdn\sum_{k=0}^n(q^{-3n};q^3)_k\ffnk{ccccc}{q}{x,q^{-\ell}/x}
{q,-1,q^{1/2-\ell},-q^{1/2},q^{-3n}}_kq^k
\\&&\xqdn\:\:=\:\sum_{k=0}^n(q^{-3n};q^3)_k\ffnk{ccccc}{q}{x,q^{-\ell}/x}
{q,-1,q^{1/2-\ell},-q^{1/2},q^{-3n}}_kq^k
\ffnk{ccccc}{q}{q^{-\ell}/x,q^{1/2-\ell}/x}{q^{-\ell}/x^2,q^{1/2-\ell+k}}_{\ell}
\\&&\xqdn\:\:\times\:
\sum_{i=0}^{\ell}q^{(\ell+\frac{1}{2})i}\frac{1-x^2q^{2i}}{1-x^2}
\ffnk{ccccc}{q}{q^{-\ell},xq^k,x^2}{q,qx,x^2q^{\ell+1}}_i
\ffnk{ccccc}{q}{q^{k-\ell}/x}{q^{-\ell}/x}_{\ell-i}.
 \enm
Interchange the summation order for the last double sum to achieve
  \bnm
&&\xxqdn\sum_{k=0}^n(q^{-3n};q^3)_k\ffnk{ccccc}{q}{x,q^{-\ell}/x}
{q,-1,q^{1/2-\ell},-q^{1/2},q^{-3n}}_kq^k
\\&&\xxqdn\:\:=\ffnk{ccccc}{q}{q^{-\ell}/x,q^{1/2-\ell}/x}{q^{-\ell}/x^2,q^{1/2-\ell}}_{\ell}
\sum_{i=0}^{\ell}q^{(\ell+\frac{1}{2})i}\frac{1+xq^{i}}{1+x}
\ffnk{ccccc}{q}{q^{-\ell},x^2}{q,x^2q^{\ell+1}}_i
\\&&\xxqdn\:\:\times\:
\sum_{k=0}^n(q^{-3n};q^3)_k\ffnk{ccccc}{q}{xq^i,q^{-i}/x}
{q,-1,q^{1/2},-q^{1/2},q^{-3n}}_kq^k.
 \enm
 Calculating the series on the last line by
Corollary \ref{corl-a}, we attain Theorem \ref{thm-c} to complete
the proof.
\end{proof}

\begin{corl}[$\ell=1$ in Theorem \ref{thm-c}]\label{corl-f}
 \bnm
&&\xqdn\sum_{k=0}^n(q^{-3n};q^3)_k\ffnk{ccccc}{q}{x,1/qx}
{q,-1,q^{-1/2},-q^{1/2},q^{-3n}}_kq^k
\\&&\xqdn\:\:=
\frac{1-qx}{2(1-q^{1/2})(1+xq^{1/2})}\ffnk{ccccc}{q^3}{qx,q^{2}/x}{q,q^2}_{n}
 \\&&\xqdn\:\:+\:\frac{1}{2}\ffnk{ccccc}{q^3}{q^2x,q/x}{q,q^2}_{n}
 -\frac{q^{1/2}(1-x)}{2(1-q^{1/2})(1+xq^{1/2})}\ffnk{ccccc}{q^3}{q^3x,1/x}{q,q^2}_{n}.
 \enm
\end{corl}

Performing the replacement $a\to q^{3x}$ in Theorem \ref{thm-c} and
then letting $q\to1$, we obtain the following equation.

\begin{prop}\label{prop-d}
For a nonnegative integer $\ell$ and a complex number $x$, there
holds
 \bnm
&&\:\xxqdn_3F_2\ffnk{cccc}{\frac{3}{4}}{3x,&-\ell-3x,&-n}{&\frac{1}{2}-\ell,&-3n}
=\frac{1}{2}\fnk{ccccc}{1+3x,\frac{1}{2}+3x}{1+6x,\frac{1}{2}}_{\ell}
\\&&\:\xxqdn\:\:\times\:
\sum_{i=0}^{\ell}\fnk{ccccc}{-\ell,6x}{1,1+6x+\ell}_i
\Bigg\{\fnk{ccccc}{\frac{1+i}{3}+x,&\frac{2-i}{3}-x}{\frac{1}{3},&\frac{2}{3}}_{n}
+\fnk{ccccc}{\frac{2+i}{3}+x,&\frac{1-i}{3}-x}{\frac{1}{3},&\frac{2}{3}}_{n}\Bigg\}.
 \enm
\end{prop}

\begin{corl}[$\ell=1$ in Proposition \ref{prop-d}]\label{corl-g}
\bnm
&&_3F_2\ffnk{cccc}{\frac{3}{4}}{3x,&-1-3x,&-n}{&-\frac{1}{2},&-3n}\\
&&\:\:=\:\frac{1+3x}{2}\fnk{ccccc}{\frac{1}{3}+x,&\frac{2}{3}-x}{\frac{1}{3},&\frac{2}{3}}_{n}
+\frac{1}{2}\fnk{ccccc}{\frac{2}{3}+x,&\frac{1}{3}-x}{\frac{1}{3},&\frac{2}{3}}_{n}
 -\frac{3x}{2}\fnk{ccccc}{1+x,&-x}{\frac{1}{3},&\frac{2}{3}}_{n}.
\enm
\end{corl}

\begin{thm}\label{thm-d}
For a nonnegative integer $\ell$ and a complex number $x$, there
holds
 \bnm
&&\xqdn\sum_{k=0}^n(q^{-3n};q^3)_k\ffnk{ccccc}{q}{x,q^{1-\ell}/x}
{q,-q^{1-\ell},q^{1/2},-q^{3/2},q^{-1-3n}}_kq^k
\\&&\xqdn\:\:=\:\ffnk{ccccc}{q}{x,-x}{x^2,-1}_{\ell}\sum_{i=0}^{\ell}(-1)^iq^{\ell i}\frac{1-x^2q^{2i-1}}{(1-x^2q^{-1})(1-q^{1/2})}
\ffnk{ccccc}{q}{q^{-\ell},x^2q^{-1}}{q,x^2q^{\ell}}_i
\\&&\xqdn\:\:\times\:\Bigg\{\frac{1-xq^i}{1+xq^{i-1/2}}\ffnk{ccccc}{q^3}{q^{3+i}x,q^{3-i}/x}{q^2,q^4}_{n}
-\frac{q^{1/2}(1-xq^{i-1})}{1+xq^{i-1/2}}\ffnk{ccccc}{q^3}{q^{2+i}x,q^{4-i}/x}{q^2,q^4}_{n}\Bigg\}.
 \enm
\end{thm}

\begin{proof}
In accordance with \eqref{series}, we gain the following relation:
  \bnm
&&\xqdn\sum_{k=0}^n(q^{-3n};q^3)_k\ffnk{ccccc}{q}{x,q^{1-\ell}/x}
{q,-q^{1-\ell},q^{1/2},-q^{3/2},q^{-1-3n}}_kq^k
\\&&\xqdn\:\:=\:\sum_{k=0}^n(q^{-3n};q^3)_k\ffnk{ccccc}{q}{x,q^{1-\ell}/x}
{q,-q^{1-\ell},q^{1/2},-q^{3/2},q^{-1-3n}}_kq^k
\ffnk{ccccc}{q}{q^{1-\ell}/x,-q^{1-\ell}/x}{q^{1-\ell}/x^2,-q^{1-\ell+k}}_{\ell}
\\&&\xqdn\:\:\times\:
\sum_{i=0}^{\ell}(-1)^iq^{\ell i}\frac{1-x^2q^{2i-1}}{1-x^2q^{-1}}
\ffnk{ccccc}{q}{q^{-\ell},xq^k,x^2q^{-1}}{q,x,x^2q^{\ell}}_i
\ffnk{ccccc}{q}{q^{1-\ell+k}/x}{q^{1-\ell}/x}_{\ell-i}.
 \enm
Interchange the summation order for the last double sum to achieve
  \bnm
&&\xxqdn\sum_{k=0}^n(q^{-3n};q^3)_k\ffnk{ccccc}{q}{x,q^{1-\ell}/x}
{q,-q^{1-\ell},q^{1/2},-q^{3/2},q^{-1-3n}}_kq^k
\\&&\xxqdn\:\:=\ffnk{ccccc}{q}{q^{1-\ell}/x,-q^{1-\ell}/x}{q^{1-\ell}/x^2,-q^{1-\ell}}_{\ell}
\sum_{i=0}^{\ell}(-1)^iq^{\ell i}\frac{1-x^2q^{2i-1}}{1-x^2q^{-1}}
\ffnk{ccccc}{q}{q^{-\ell},x^2q^{-1}}{q,x^2q^{\ell}}_i
\\&&\xxqdn\:\:\times\:
\sum_{k=0}^n(q^{-3n};q^3)_k\ffnk{ccccc}{q}{xq^i,q^{1-i}/x}
{q,-q,q^{1/2},-q^{3/2},q^{-1-3n}}_kq^k.
 \enm
 Evaluating the series on the last line by the equivalent form of
 Corollary \ref{corl-c}:
 \bnm
&&\xqdn\sum_{k=0}^n(q^{-3n};q^3)_k\ffnk{ccccc}{q}{x,q/x}
{q,-q,q^{1/2},-q^{3/2},q^{-1-3n}}_kq^k
\\&&\xqdn\:\:=
\frac{1-x}{(1-q^{1/2})(1+xq^{-1/2})}\ffnk{ccccc}{q^3}{q^3x,q^{3}/x}{q^2,q^4}_{n}
\\&&\xqdn\:\:+\:
\frac{1-xq^{-1}}{(1-q^{-1/2})(1+xq^{-1/2})}\ffnk{ccccc}{q^3}{q^2x,q^{4}/x}{q^2,q^4}_{n},
 \enm
we attain Theorem \ref{thm-d} to finish the proof.
\end{proof}

\begin{corl}[$\ell=1$ in Theorem \ref{thm-d}]\label{corl-h}
  \bnm
&&\xqdn\sum_{k=0}^n(q^{-3n};q^3)_k\ffnk{ccccc}{q}{x,1/x}
{q,-1,q^{1/2},-q^{3/2},q^{-1-3n}}_kq^k
\\&&\xqdn\:\:=
\frac{1-qx}{2(1-q^{1/2})(1+xq^{1/2})}\ffnk{ccccc}{q^3}{q^4x,q^2/x}{q^2,q^4}_{n}
 \\&&\xqdn\:\:+\:
 \frac{(1-x)^2}{2(1+xq^{-1/2})(1+xq^{1/2})}\ffnk{ccccc}{q^3}{q^3x,q^3/x}{q^2,q^4}_{n}
 \\&&\xqdn\:\:-\:
 \frac{q^{1/2}(1-xq^{-1})}{2(1-q^{1/2})(1+xq^{-1/2})}\ffnk{ccccc}{q^3}{q^2x,q^4/x}{q^2,q^4}_{n}.
 \enm
\end{corl}

Employing the substitution $a\to q^{3x}$ in Theorem \ref{thm-d} and
then letting $q\to1$, we obtain the following equation.

\begin{prop}\label{prop-e}
For a nonnegative integer $\ell$ and a complex number $x$, there
holds
 \bnm
&&\:\xxqdn_3F_2\ffnk{cccc}{\frac{3}{4}}{3x,&1-\ell-3x,&-n}{&\frac{1}{2},&-1-3n}
\\&&\:\xxqdn\:\:=\:
\fnk{ccccc}{3x}{6x}_{\ell}\sum_{i=0}^{\ell}(-1)^i\frac{6x+2i-1}{6x-1}\fnk{ccccc}{-\ell,6x-1}{1,6x+\ell}_i
\\&&\:\xxqdn\:\:\times\:
\Bigg\{(3x+i)\fnk{ccccc}{\frac{3+i}{3}+x,&\frac{3-i}{3}-x}{\frac{2}{3},&\frac{4}{3}}_{n}
-(3x+i-1)\fnk{ccccc}{\frac{2+i}{3}+x,&\frac{4-i}{3}-x}{\frac{2}{3},&\frac{4}{3}}_{n}\Bigg\}.
 \enm
\end{prop}

\begin{corl}[$\ell=1$ in Proposition \ref{prop-d}]\label{corl-i}
\bnm
&&_3F_2\ffnk{cccc}{\frac{3}{4}}{3x,&-3x,&-n}{&\frac{1}{2},&-1-3n}\\
&&\:\:=\:\frac{1+3x}{2}\fnk{ccccc}{\frac{4}{3}+x,&\frac{2}{3}-x}{\frac{2}{3},&\frac{4}{3}}_{n}
+\frac{1-3x}{2}\fnk{ccccc}{\frac{2}{3}+x,&\frac{4}{3}-x}{\frac{2}{3},&\frac{4}{3}}_{n}.
\enm
\end{corl}

With the change of the parameter $\ell$, these theorems and
propositions can create more concrete formulas. They will not be
laid out in this paper.

The case $n\to \infty$ of \eqref{gosper-b} reads as
 \bnm
_2F_1\ffnk{cccc}{\frac{1}{4}}{3x,&2-3x}{&\frac{3}{2}}
=\frac{\Gamma(2/3)\Gamma(4/3)}{\Gamma(2/3+x)\Gamma(4/3-x)}.
 \enm
Letting $x=1/3$ in this equation, we recover the beautiful series
for $\pi$ (cf. \citu{weisstein}{Equation (27)}):
 \bnm
 \sum_{k=0}^{\infty}\frac{k!}{(3/2)_k}\frac{1}{4^k}=\frac{2\pi}{3\sqrt{3}}.
 \enm
Unfortunately, we don't find new series for $\pi$ from the family of
summation formulas for $_3F_2(\frac{3}{4})$-series.

\textbf{Acknowledgments}

The work is supported by the Natural Science Foundations of China
(Nos. 11301120, 11201241 and 11201291).


\end{document}